\documentclass[11pt,reqno]{amsart}

\usepackage[dvipsnames]{xcolor}
\usepackage{tikz}
\usepackage{stmaryrd}
\usetikzlibrary{matrix,arrows,decorations.pathmorphing}
\usepackage{graphicx}
\usepackage{amssymb}
\usepackage{amsmath,mathtools}
\usepackage{fullpage}
\usepackage{caption}
\usepackage{amsthm}
\usepackage{mathrsfs}
\usepackage{thmtools}
\usepackage{blkarray}
\usepackage{multirow}
\usepackage{picinpar} 
\usepackage{tikz-cd}
\usepackage{color}
\usepackage{verbatim}
\usepackage{hyperref}
\hypersetup{colorlinks=true, linkcolor=red, citecolor = OliveGreen, urlcolor = black}
\usepackage{amssymb}
\usepackage{amsmath}
\usepackage{enumitem,colonequals}

\usepackage{color}
\usepackage{mathrsfs}
\usepackage[all]{xy}
\usepackage{comment}
\usepackage{mathpazo}
\usepackage{euler}

\newcommand{\mC}{\mathbb{C}}
\newcommand{\mG}{\mathbb{G}}

\newcommand{\cO}{\mathcal{O}}

\newcommand{\wt}[1]{\widetilde{#1}}

\RequirePackage{mathrsfs} 

\numberwithin{equation}{subsection}
\newtheorem{thmx}{Theorem}

\newtheorem{corox}[thmx]{Corollary}

\newtheorem{theorem}{Theorem}[section]
\newtheorem{lemma}[theorem]{Lemma}
\newtheorem{proposition}[theorem]{Proposition}

\theoremstyle{definition}

\newtheorem{remark}[theorem]{Remark}

\newtheorem{example}[theorem]{Example}
\newtheorem{definition}[theorem]{Definition}

\numberwithin{equation}{section} \numberwithin{figure}{section}

\DeclareMathOperator{\an}{an}

\DeclareMathOperator{\Pic}{Pic}

\DeclareMathOperator{\Sp}{Sp}

\newcommand\QQ{\mathbb{Q}}

\newcommand\CC{\mathbb{C}}

\renewcommand{\geq}{\geqslant}

\title[Closed subvarieties of semi-abelian varieties]{Non-archimedean entire curves in closed subvarieties of semi-abelian varieties}
 
\author{Jackson S. Morrow} 
\address{Jackson S. Morrow \\
	Department of Mathematics\\
	Emory University\\
    Atlanta, GA 30322\\
    United States}
\email{jmorrow4692@gmail.com}



\begin{document}

\begin{abstract}  
We prove a non-archimedean analogue of the fact that a closed subvariety of a semi-abelian variety is  hyperbolic modulo its special locus, and thereby generalize a result of Cherry.
\end{abstract}

\maketitle

\thispagestyle{empty}

\section{Introduction}  
  
The Green--Griffiths--Lang--Vojta conjectures predict that a quasi-projective variety $X$ over $\CC$ is of log-general type if and only if there is a proper closed subscheme $\Delta\subset X$ such that $X$ is Brody hyperbolic modulo $\Delta$ (i.e., every non-constant holomorphic map $\CC\to X(\mC)$ factors through $\Delta(\mC)$); see \cite{Br1, Br2, GrGr, Lang, VojtaLangExc, RousseauNotes}.  For example, this conjecture is known when  $X$  is a closed subvariety of an abelian variety by the celebrated theorem of Bloch--Ochiai--Kawamata \cite{Bloch26, Ochiai77, Kawamata}. We refer the reader to \cite{NWY1, NWY2} for more recent advances.  

  The aim of this paper is to investigate non-archimedean analogues of the Green--Griffiths--Lang--Vojta conjectures. Our starting point is the following theorem, which is the culmination of results in \cite{FaltingsLang1, FaltingsLang2, Abram, Vojta11, Nogu}. 
  The definitions of the notions appearing in the following theorem are  stated in \cite[p.~78]{LangComplexHyperbolic} and \cite[Definitions 7.1, 8.1]{JBook}. 

  \begin{theorem}[Abramovich, Faltings, Kawamata, Noguchi, Ueno, Vojta] \label{thm:starting_point}
  Let $X$ be a closed subvariety of a semi-abelian variety $G$ over $\CC$. Let $\Sp(X)$ be the union of the subvarieties of $X$ which are translates of positive-dimensional  closed subgroups of $G$. Then the following statements hold.
  \begin{enumerate}
  \item The subset $\Sp(X)$ is  Zariski closed in $X$.
  \item The variety $X$ is of log-general type if and only if $\Sp(X) \neq X$.
  \item The variety $X$ is arithmetically hyperbolic modulo $\Sp(X)$.
  \item The variety $X$ is Brody hyperbolic modulo $\Sp(X)$.
  \end{enumerate}
  \end{theorem}

In \cite{Cherry, CherryKoba, CherryRu, AnCherryWang, LinWang, LevinWang}, the authors investigate   possible non-archimedean analogues of the Green--Griffiths--Lang conjecture; however, some of their results contrast the complex analytic setting. Inspired by Cherry's work, the authors of \cite{JVez} formulated the ``correct" analogue of the Green--Griffiths--Lang conjecture for projective varieties over a non-archimedean valued field $K$.

Our main result is the non-archimedean analogue of the statements $(2), (3)$, and $(4)$ in Theorem \ref{thm:starting_point}. We refer the reader to Section \ref{section:defs} for the definition of a $K$-analytic Brody hyperbolic variety.

\begin{thmx}\label{thm:GGLSemiAb}
Let $K$ be an algebraically closed complete non-archimedean valued field of characteristic zero. 
Let $X$ be a closed subvariety of a semi-abelian variety $G$ over $K$. Then $X$ is  $K$-analytically Brody hyperbolic modulo $\Sp(X)$.
\end{thmx}

This result is proven by  Cherry \cite[Theorem 3.5]{Cherry} when $G$ is an abelian variety (so that $X$ is projective). Our line of reasoning to prove  Theorem \ref{thm:GGLSemiAb} resembles Cherry's in that we study analytic morphisms from tori to semi-abelian varieties (see Section \ref{section:extension}).

A direct consequence of Theorem \ref{thm:GGLSemiAb} is the following characterization of groupless (\cite[Definition 2.1]{JKam}) closed subvarieties of a semi-abelian variety.
\begin{corox}
Let $K$ be an algebraically closed complete non-archimedean valued field of characteristic zero, and let $X$ be a closed subvariety of a semi-abelian variety $G$ over $K$.
Then  $X$ is groupless (i.e., does not contain the translate of a positive-dimensional closed subgroup of $G$) if and only if $X$ is $K$-analytically Brody hyperbolic. 
\end{corox}

\subsection*{Acknowledgments}
The author would like to thank Ariyan Javanpeykar for suggesting the problem, for many helpful conversations, and for sending a preliminary version of \cite{JBook}, which helped in writing Section \ref{section:defs}. 
The author extends his thanks to Alberto Vezzani and David Zureick-Brown for useful discussions.
The author also thanks Lea Beneish for comments on an earlier draft.

\subsection*{Conventions}
Throughout, $K$ denotes an algebraically closed complete non-archimedean valued field of characteristic zero. 
For a locally of finite type scheme $X/K$, we will use $X^{\an}$ to denote the rigid analytic space (in the sense of Tate \cite{TateRigid}) or the adic space (in the sense of Huber \cite{huber2}) associated to $X$, unless otherwise stated.

\section{Non-archimedean hyperbolicity of pairs}\label{section:defs}
In this section, we extend the notion of $K$-analytic Brody hyperbolicity of varieties introduced in \cite{JVez}.
First, we recall their notion.

\begin{definition}[\protect{\cite[Definition 2.3, Lemma 2.14, Lemma 2.15]{JVez}}]
Let $X$ be a finite type separated scheme over $K$. Then $X$ is \emph{$K$-analytically Brody hyperbolic} if
\begin{itemize}
\item every analytic morphism $\mathbb{G}_{m,K}^{\an} \to X^{\an}$ is constant, and 
\item for every abelian variety $A$ over $K$, every   morphism $A\to X$ is constant.
\end{itemize}
\end{definition}

We now define what it means for a pair $(X,D)$ to be hyperbolic. Our proposed definition reads as follows.

\begin{definition} \label{defn1}
Let $X$ be a finite type separated scheme  over $K$  and let $D\subset X$ be a closed subscheme. Then $X$ is \emph{$K$-analytically Brody hyperbolic modulo $D$} (or:~\emph{the pair $(X,D)$ is $K$-analytically Brody hyperbolic)} if
\begin{itemize}
\item every non-constant analytic morphism $\mathbb{G}_{m,K}^{\an} \to X^{\an}$ factors over $D^{\an}$, and
\item for every abelian variety $A$ over $K$ and every dense open subset $U\subset A$ with $\mathrm{codim}(A\setminus U)\geq 2$, every non-constant morphism $U \to X$ of schemes factors over $D$. 
\end{itemize}
\end{definition}

With this definition, it is not hard to see that a proper scheme $X$ over $K$ is $K$-analytically Brody hyperbolic if and only if $(X,\emptyset)$ is $K$-analytically Brody hyperbolic (i.e., $X$ is $K$-analytically Brody hyperbolic modulo the empty set). Indeed, if $X$ is $K$-analytically Brody hyperbolic and proper, then $X$ has no rational curves. In particular, every   rational map $A\dashrightarrow X$ with $A$ an abelian variety 
extends (uniquely) to a morphism $A\to X$, and such morphisms are constant if $X$ is $K$-analytically Brody hyperbolic.

Similarly, one can show that a closed subscheme $X$ of a semi-abelian variety is $K$-analytically Brody hyperbolic if and only if $(X,\emptyset)$ is $K$-analytically Brody hyperbolic; see Remark \ref{remark:mochi}.

\begin{remark} 
The reader might find the condition in Definition \ref{defn1} on the codimension of the complement of $U$ in $A$ unnatural. We now explain why this condition is necessary (assuming one wants to define the ``right'' notion of hyperbolicity).
Note that Vojta has already made the observation that one has to test hyperbolicity on ``big'' open subsets of algebraic groups and not merely on algebraic groups; see \cite[Definition 2.2]{VojtaLangExc} and also \cite[\S 6]{JBook}.

The example to keep in mind is the blow-up $X$ of a simple abelian surface $A$ at the origin over $\QQ$. It is not hard to see that $X(\CC)$ admits a dense entire curve, and is therefore as far as possible from being  Brody hyperbolic (in the usual complex-analytic sense).

Let $p$ be a prime of good reduction of $A$, and consider the smooth projective variety $X_{\CC_p}$ over $\CC_p$.  Let $\Delta$ be the exceptional locus of $X_{\CC_p}\to A_{\CC_p}$. Then, every non-constant morphism from $\mathbb{G}_m^{\an}\to X_{\CC_p}^{\an}$ factors over $\Delta^{\an}$. Moreover, by rigid analytic GAGA \cite{kopfGAGA} and the simplicity of $A_{\CC_p}$, for every abelian variety $B$ over $\CC_p$, every morphism $B^{\an}\to X_{\CC_p}^{\an}$ is constant. Thus, if one does not ``test'' the hyperbolicity on big opens of abelian varieties, the variety $X_{\CC_p}$ would be $\CC_p$-analytically Brody hyperbolic modulo $\Delta$ (contrary to it being very far from being Brody hyperbolic over $\CC$).
\end{remark}

\begin{remark} \label{remark:mochi}
Let $A/K$ be an abelian variety and let $G/K$ be a semi-abelian variety. 
By  \cite[Lemma A.2]{MochizukiAbsoluteAnabelian},  for every dense open subset $U\subset A$ with $\mathrm{codim}(A\setminus U)\geq 2$, we have that any morphism $U\to G$ extends uniquely to a morphism $A\to G$. 
Using this result, we immediately have that a closed subscheme $X$ of $G$ is $K$-analytically Brody hyperbolic if and only if $X$ is $K$-analytically Brody modulo $\emptyset$.
\end{remark}

\begin{definition}
A finite type separated scheme $X$ over $K$ is \emph{pseudo-$K$-analytically Brody hyperbolic} if there is a proper closed subset $D \subsetneq X$ of $X$ such that $(X,D)$ is $K$-analytically Brody hyperbolic. 
\end{definition}

\section{Analytic maps from tori to semi-abelian varieties}\label{section:extension}
Let $G$ be a semi-abelian variety over $K$. 
Since $G$ is semi-abelian, there is a split torus $T_1 \subset G$, an abelian variety $A$ over $K$, and a short exact sequence of commutative group schemes
\[
0\to T_1\to G\to A\to 0.
\]

Our goal is to prove that, if   $\phi\colon \mathbb{G}_m^{\an}\to G^{\an}$ is a morphism, then the Zariski closure of its image is the translate of the analytification of an \textit{algebraic} subgroup of $G$; see Proposition \ref{prop:ZCAlgebraicGroup} for a precise statement.

\begin{remark}
In this section, for a locally of finite type scheme $X/K$, we will use $X^{\an}$ to denote the associated $K$-analytic space (in the sense of Berkovich \cite{BerkovichSpectral}). 
We do so in order to use techniques from topology to study analytic maps from tori to semi-abelian varieties.
In particular, an analytic torus is simply-connected \cite[Section 6.3]{BerkovichSpectral}, and a famous result of Berkovich \cite[Corollary 9.5]{BerkovichUniversalCover} states that a smooth, connected, Hausdorff strictly $K$-analytic space has a universal covering which is a Hausdorff, simply connected strictly $K$-analytic space.

We can relate our results concerning Berkovich spaces to adic spaces using the equivalence between the category of Hausdorff strictly $K$-analytic spaces and the category of taut locally of finite type adic spaces \cite[Proposition 8.3.7]{huber}.
\end{remark}

We start by recalling that line bundles on analytifications of tori are trivial.

\begin{lemma}\label{lemma:BerkrigidPic}
Let $X$ be a separated, good, strictly $K$-analytic space, and let $X_0$ denote the associated rigid analytic space. Then $\Pic(X) \cong \Pic(X_0)$.
\end{lemma}

\begin{proof}
This follows from \cite[Corollary 1.3.5]{BerkovichEtaleCohomology} and the bottom of \textit{loc.~cit.~}p.~37.
\end{proof}

\begin{lemma}\label{thm:PicTAnTrivial} 
If $L$ is a line bundle on a split torus $\mathbb{G}_{m,K}^{r,\an}$, then $L$ is trivial.  
\end{lemma}
\begin{proof}
Since the Berkovich analytification of $\mG_{m,K}^r$ is a separated, good, strictly $K$-analytic space, our result follows from \cite[Theorem 6.3.3.(2)]{FresnelVDPutRigidAnalytic} and Lemma \ref{lemma:BerkrigidPic}.
\end{proof}

\begin{lemma}\label{prop:ZCinTorus}
Let $\phi\colon \mathbb{G}_m^{\an}\to G^{\an}$ be a morphism, and let $\wt{\phi}\colon \mathbb{G}_m^{\an} \to \wt{G}$ be a lift of this morphism to the universal cover $\wt{G}$ of $G^{\an}$. Then, the image $\wt{\phi}(\mathbb{G}_m^{\an})$ is contained inside a split torus $T^{\an}$ of $\wt{G}$. 
\end{lemma}

\begin{proof}
Let $\widetilde{A}$ be the universal covering of $A^{\an}$. By \cite[Uniformization Theorem 8.8]{BLII}, there is  a semi-abelian variety $H$ over $K$ with $\widetilde{A} \cong H^{\an}$, an abelian variety $B$ over $K$ with good reduction over $\cO_K$, a split torus $T_2\subset H$, and a short exact sequence of commutative group schemes
\[
0\to T_2\to H\to B\to 0.
\]  
Let $\widetilde{G}$ be the universal covering space of $G^{\an}$. Note that there is a   structure of a commutative Berkovich analytic group on $\widetilde{G}$  which makes $\widetilde{G}\to G^{\an}$ into a homomorphism.  By the universal property of universal covering spaces, the surjective homomorphism $G^{\an}\to A^{\an}$ lifts uniquely to a homomorphism $\widetilde{G}\to H^{\an}$.  

The image of $\mathbb{G}_m^{\an}\to \wt{G}\to H^{\an}$ is contained in $T_2^{\an}$. 
Indeed, the morphism $\mathbb{G}_m^{\an}\to H^{\an} \to B^{\an}$ is constant, since $B^{\an}$ has good reduction \cite[Theorem 3.2]{Cherry}, and so the image of $\mathbb{G}_m^{\an}$ in $H^{\an}$ lands inside its torus $T_2^{\an}$ (up to translation).

Since $T_1^{\an}$ is simply-connected \cite[Section 6.3]{BerkovichSpectral}, the subgroup $T_1^{\an}\subset G^{\an}$ lifts to a subgroup $T_1^{\an}\subset \widetilde{G}$.  Note that the homomorphism $T_1^{\an}\to H^{\an}$  factors over $T_{2}^{\an}$, and that the morphism $T_1^{\an}\to T_2^{\an}$ is algebraic \cite[Proposition 3.4]{Cherry}, i.e., the analytification of some morphism $T_1\to T_2$. Let $T_3$ be the image of this morphism, which is again a split torus.

Let $F$ be the inverse image of $T_3^{\an} \subset H^{\an}$ in $\widetilde{G}$. Note that $F$ is a closed subgroup of $\widetilde{G}$  and that the kernel of the homomorphism $F\to T_3^{\an}$ equals $T_1^{\an}$. Thus, there is a short exact sequence of rigid analytic groups
\[ 0\to T_1^{\an}\to F\to T_3^{\an}\to 0.\] 
By Lemma \ref{thm:PicTAnTrivial}, the above sequence splits, and so $F$ is the analytification of a split torus $T$. This shows that the image $\wt{\phi}(\mathbb{G}_m^{\an})$ is contained inside the split torus $T^{\an}$, as required. 
\end{proof}

\begin{lemma}\label{prop:absimple}
Let $\phi\colon \mathbb{G}_m^{\an}\to G^{\an}$ be a morphism. Suppose that the image of $\mathbb{G}_m^{\an}\to G^{\an} \to A^{\an}$ is Zariski dense. Then, the image of $\phi$ is Zariski dense in $G^{\an}$.
\end{lemma}

\begin{proof}
Lemma \ref{prop:ZCinTorus} asserts that $\phi(\mathbb{G}_m^{\an})$ is an analytic subgroup $F'$ of $G^{\an}$, as it is the composition of group homomorphisms $\wt{\phi}$ and the uniformization map, which is an analytic group homomorphism. Since $\mathbb{G}_m^{\an}\to G^{\an} \to A^{\an}$ is Zariski dense, $F'$ dominates $A^{\an}$, and this analytic group homomorphism has kernel $T_1^{\an}$. Moreover, we have the following morphism of short exact sequences of analytic groups:
\[
\begin{tikzcd}
0\to T_1^{\an} \arrow{r}\arrow[equal]{d} & F' \arrow{r}\arrow{d}{f} & A^{\an} \arrow{r}\arrow[equal]{d} & 0 \\
0\to T_1^{\an} \arrow{r} & G^{\an} \arrow{r} & A^{\an} \arrow{r} & 0. 
\end{tikzcd}
\]
By Berkovich analytic GAGA \cite[Corollary 3.4.10]{BerkovichSpectral}, $\Pic(A^{\an})$ is in bijective correspondence with $\Pic(A)$, which implies that $F'$ is in fact an algebraic subgroup of $G^{\an}$.
Moreover,  the short five lemma tells us that the morphism $f$ must be an isomorphism.
\end{proof}

\begin{proposition}\label{prop:ZCAlgebraicGroup}
Let $\phi\colon \mathbb{G}_m^{\an}\to G^{\an}$ be a morphism.   Then, the Zariski closure of $\phi(\mathbb{G}_m^{\an})$ in $G^{\an}$ is the analytification of the translate of  of an algebraic subgroup of $G$. 
\end{proposition}
\begin{proof} Let $\psi\colon\mathbb{G}_m^{\an}\to A^{\an}$ be the composition of $\phi$ and the surjective homomorphism $G^{\an}\to A^{\an}$. By Lemma \ref{prop:ZCinTorus}, the image $\phi(\mathbb{G}_m^{\an})$ is an analytic subgroup of $G^{\an}$. Therefore, the image $\psi(\mathbb{G}_m^{\an})$ is an analytic subgroup of $A^{\an}$. Thus,  the Zariski closure of the image of   $\psi$ is  an abelian subvariety $E^{\an}$ of $A^{\an}$ (see \cite[Proof of Theorem 3.6]{Cherry}). 

Now, let $F$ be the preimage of $E$ inside $G$, and note that $F$ is a semi-abelian variety (as it is a closed subgroup of $G$). Clearly, the image of the morphism $\phi\colon\mathbb{G}_m^{\an}\to G^{\an}$ is contained in $F^{\an}$. Now, by construction, the image of the composed morphism $\mathbb{G}_m^{\an}\to F^{\an}\to E^{\an}$ is Zariski dense in $E$. Therefore, by Lemma \ref{prop:absimple}, the image of $\mathbb{G}_m^{\an}$ in $F^{\an}$ is the analytification of  the translate of an algebraic subgroup of $F$. In particular, it is the analytification of the translate of an algebraic subgroup of $G$.
\end{proof}

The following example shows that the image of an algebraic group under an analytic homomorphism is not necessarily an algebraic subgroup.

\begin{example}
Let $E/K$ be an elliptic curve with multiplicative reduction and let $\phi\colon\mathbb{G}_{m,K}^{\an}\to E^{\an}$ be the universal covering of $E^{\an}$. 
Consider the semi-abelian variety  $G= \mathbb{G}_{m,K} \times E $. 
Let $\mathbb{G}_{m,K}^{\an}\to \mathbb{G}_{m,K}^{\an}\times E^{\an}$ be the morphism defined by $z\mapsto (z,\phi(z))$. The image of this morphism is \textit{not} an algebraic subgroup of $G^{\an}$. However, its Zariski closure equals $G^{\an}$.
\end{example}

To end this section, we prove Theorem \ref{thm:GGLSemiAb}.

\begin{proof}[Proof of Theorem \ref{thm:GGLSemiAb}] 
Proposition \ref{prop:ZCAlgebraicGroup} tells us that the Zariski closure of every analytic morphism $\mG_m^{\an} \to X^{\an} \subset G^{\an}$ is contained in $\Sp(X)^{\an}$. 
To conclude the proof, it suffices to show that for every abelian variety $A$ over $K$ and every dense open subset $U\subset A$ with $\mathrm{codim}(A\setminus U)\geq 2$, we have that every non-constant morphism $U \to X$ of schemes factors over $\Sp(X)$. 
By Remark \ref{remark:mochi}, every morphism $U \to X \subset G$ extends to a morphism $A \to X \subset G$. 
Now, by \cite[Theorem 2]{IitakaLogForms},    any morphism between semi-abelian varieties is the composition of a group homomorphism and a translation,   so that the image of $A \to X \subset G$ factors over $\Sp(X)$, as desired.
\end{proof}

  \bibliography{refs}{}
\bibliographystyle{alpha}

 \end{document}